\def\BibCommonsStandalone{}

\documentclass[11pt]{article}
\usepackage{adjustbox,aliascnt,amsmath,amssymb,amsthm,calc,colonequals,enumitem,mathtools,mleftright,parskip,stmaryrd,textgreek,xcolor}
\usepackage[a4paper,left=1.3in,right=1.3in,top=1.25in,bottom=1.25in,footskip=0.6in]{geometry}
\usepackage[unicode]{hyperref}
\usepackage[en-GB]{datetime2}

\usepackage{bib-commons}

\newcommand{\newaliastheorem}[2]{%
  \newaliascnt{#1}{theorem}
  \newtheorem{#1}[#1]{#2}
  \aliascntresetthe{#1}
  \expandafter\def\csname #1autorefname\endcsname{#2}
}

\theoremstyle{plain}
\newtheorem{theorem}{Theorem}[section]
\newaliastheorem{lemma}{Lemma}
\newaliastheorem{proposition}{Proposition}
\newaliastheorem{corollary}{Corollary}

\newcommand{\tuple}[1]{\left\langle #1 \right\rangle}

\newcommand{\Ord}{\mathrm{Ord}}
\newcommand{\PlOrd}{\mathrm{PlOrd}}

\newcommand{\pl}{\mathrm{pl}}
\newcommand{\pls}{{\pl +}}

\DeclareMathOperator{\dd}{def}

\DeclareMathOperator{\relpl}{relpl}

\title{Some notes on plump ordinals}
\author{Shuwei Wang}
\date{\DTMdate{2026-01-30}}

\begin{document}

\maketitle

\begin{abstract}
  In this exposition, we attempt to formalise a treatment of Paul Taylor's notion of plump ordinals \cite{taylor96-intuitionistic-ordinals} in weak intuitionistic axiomatic set theories such as $\mathrm{IKP}$. We will explore basic properties of plump ordinals, especially in relation to G\"odel's constructible universe $L$ and incomparable codings. As a quick application, we explain at the end how plump ordinals can be used to build a Heyting-valued model $V^\mathbb{H}$ from a classical $V \vDash \mathrm{ZFC}$ such that for some arbitrary, fixed $x \in V$ we have
  \[V^\mathbb{H} \vDash \mathcal{P}\mleft(\check{x}\mright) \in L.\]
\end{abstract}

\section{Introduction}
\label{sec:intro}

Paul Taylor introduced the notion of plump ordinals in \cite{taylor96-intuitionistic-ordinals} as a distinguished sub-class of usual (transitive) ordinals with tamer properties in intuitionistic set theory. However, his work is based in the context of category-theoretic ensembles, and does not characterise precisely how much first-order axiomatic set theory is needed to set-up such a notion (which seemingly requires recursive constructions over the powerset operator). In this exposition, we shall reconsider Taylor's idea formally in the weak base theory $\mathrm{IKP}$ which includes strong infinity, adding additional axioms explicitly where necessary.

To begin, we follow Taylor and say that a class $\mathcal{O} \subseteq \Ord$ is \emph{the class of plump ordinals} if for any $\alpha \in \Ord$, we have $\alpha \in \mathcal{O}$ if and only if both of the following hold:
\begin{itemize}
  \item $\alpha \subseteq \mathcal{O}$,
  \item for any $\beta \in \alpha$, the class $\mathcal{P}\mleft(\beta\mright) \cap \mathcal{O} \subseteq \alpha$.
\end{itemize}

Formally speaking, this is not a definition for the class $\mathcal{O}$, because it is self-referential and some form of recursion theorem is needed to justify that such a class $\mathcal{O}$ is indeed definable. But first observe that the class of plump ordinals, if it exists, is at least unique:

\begin{proposition}
  Suppose that $\mathcal{O}_1, \mathcal{O}_2$ are both ``classes of plump ordinals'' as defined above, then $\mathcal{O}_1 = \mathcal{O}_2$.
\end{proposition}

\begin{proof}
  We shall prove by set induction that for any $\alpha \in \mathcal{O}_1 \cup \mathcal{O}_2$, the classes
  \[\mathcal{P}\mleft(\alpha\mright) \cap \mathcal{O}_1 = \mathcal{P}\mleft(\alpha\mright) \cap \mathcal{O}_2.\]

  To begin, assume that $\alpha \in \mathcal{O}_1$, and we shall first show using the inductive hypothesis that $\mathcal{P}\mleft(\alpha\mright) \cap \mathcal{O}_1 \subseteq \mathcal{P}\mleft(\alpha\mright) \cap \mathcal{O}_2$. Consider any $\beta \subseteq \alpha$ such that $\beta \in \mathcal{O}_1$, then for any $\gamma \in \beta \subseteq \alpha$, we also have $\gamma \in \mathcal{O}_1$. By the inductive hypothesis,
  \[\gamma \in \mathcal{P}\mleft(\gamma\mright) \cap \mathcal{O}_1 = \mathcal{P}\mleft(\gamma\mright) \cap \mathcal{O}_2,\]
  i.e.\ $\gamma \in \mathcal{O}_2$ as well. Therefore, $\beta \subseteq \mathcal{O}_2$.

  Additionally, for any $\gamma \in \beta \subseteq \alpha$, we have $\mathcal{P}\mleft(\gamma\mright) \cap \mathcal{O}_1 \subseteq \beta$. Thus again by the inductive hypothesis,
  \[\mathcal{P}\mleft(\gamma\mright) \cap \mathcal{O}_2 = \mathcal{P}\mleft(\gamma\mright) \cap \mathcal{O}_1 \subseteq \beta.\]
  Therefore, both criteria are verified, and by the definition of $\mathcal{O}_2$, we must also have $\beta \in \mathcal{O}_2$.

  Now, by an entirely symmetrical argument, we can show that $\alpha \in \mathcal{O}_2$ implies $\mathcal{P}\mleft(\alpha\mright) \cap \mathcal{O}_2 \subseteq \mathcal{P}\mleft(\alpha\mright) \cap \mathcal{O}_1$. Notice that as $\alpha \in \mathcal{P}\mleft(\alpha\mright)$, this suffices for the conclusion that for any $\alpha \in \mathcal{O}_1 \cup \mathcal{O}_2$, we have
  \[\alpha \in \mathcal{O}_1 \leftrightarrow \alpha \in \mathcal{O}_2.\]
  In other words, $\alpha \in \mathcal{O}_1 \cup \mathcal{O}_2$ implies both $\alpha \in \mathcal{O}_1$ and $\alpha \in \mathcal{O}_2$, and thus $\mathcal{P}\mleft(\alpha\mright) \cap \mathcal{O}_1 = \mathcal{P}\mleft(\alpha\mright) \cap \mathcal{O}_2$ as desired.

  Finally, simply observe that we have already shown $\mathcal{O}_1 = \mathcal{O}_2$ in this process.
\end{proof}

Now, using a similar recursion, it is not hard to show that such a class $\mathcal{O}$ exists as a $\Sigma_1\mleft(\mathcal{P}\mright)$-class, in the theory $\mathrm{IKP}\mleft(\mathcal{P}\mright)$ (where in addition to the axioms of $\mathrm{IKP}$, we have a primitive powerset operator $\mathcal{P}$ that can appear in the clauses of $\Delta_0$-separation and collection). Following the construction in Taylor's original proof of \cite{taylor96-intuitionistic-ordinals}*{Proposition 4.3}, we can define the class function $\vartheta_\alpha\mleft(\beta\mright)$, which maps two parameters $\alpha \in \Ord$, $\beta \in \mathcal{P}\mleft(\alpha\mright)$ to an element in the set $\Omega = \mathcal{P}\mleft(1\mright)$ of all truth values, by recursion on $\alpha$ as the following:
\[\vartheta_\alpha\mleft(\beta\mright) = \left\{\varnothing : \forall \gamma \in \beta \left(\varnothing \in \vartheta_\gamma\mleft(\gamma\mright) \land \forall \delta \subseteq \gamma \left(\varnothing \in \vartheta_\gamma\mleft(\delta\mright) \rightarrow \delta \in \beta\right)\right)\right\}.\]
Then $\mathcal{O} = \left\{\alpha \in \Ord : \vartheta_\alpha\mleft(\alpha\mright) = 1\right\}$ will be the class of plump ordinals.

However, for our purpose of studying the interaction of the plump ordinals and G\"odel's constructible universe $L$, it can be difficult to adapt this construction since the powerset axiom is known to fail in $L$ even when it holds in $V$ (cf.\ \cite{matthews-rathjen24-constructible-universe}*{section 7}), if the background theory does not have $\Sigma_1$-separation.

Thus, in this exposition we shall start by developing an alternative definition that characterise the plump ordinals without involving the powersets. The goal is to show that the class of plump ordinals, as characterised above, exists in the base theory $\mathrm{IKP}$ already. Afterwards we shall observe its basic properties, some not available to the larger class $\Ord$, especially in relation to $L$. For example, while it remains open whether any intuitionistic set theory, even the stronger ones like $\mathrm{IZF}$, prove that the class $L$ contains all ordinals, it is much simpler to show (already in $\mathrm{IKP}$) that $L$ contains all plump ordinals. We will look at how this can be useful together with arithmetic on the plump ordinals and techniques of incomparable coding.

\section{A formal definition of plumpness}

To grasp the precise formal complexity of the class of plump ordinals, we now opt to write down an explicit definition for the class. We denote
\[\relpl_\alpha\mleft(\gamma\mright) \coloneqq \forall \delta \in \gamma \ \forall \varepsilon \in \alpha \left(\varepsilon \subseteq \delta \rightarrow \varepsilon \in \gamma\right)\]
and say that \emph{$\gamma$ is plump relative to $\alpha$} when this holds. We then define the following sub-class of $\Ord$:
\[\PlOrd = \left\{\alpha \in \Ord : \forall \beta \in \alpha \ \forall \gamma \subseteq \beta \left(\relpl_\alpha\mleft(\gamma\mright) \rightarrow \gamma \in \alpha \land \forall \delta \in \alpha \left(\beta \in \delta \rightarrow \gamma \in \delta\right)\right)\right\}.\]

We observe the following basic properties:

\begin{lemma}
  \label{lem:pl-ord-elem-rel-pl}
  For any $\alpha \in \PlOrd$, if $\beta \in \alpha$, then $\relpl_\alpha\mleft(\beta\mright)$ holds.
\end{lemma}

\begin{proof}
  We show by set induction on $\beta$ that if $\beta \in \alpha$, then
  \[\forall \gamma \in \alpha \left(\gamma \subseteq \beta \rightarrow \relpl_\alpha\mleft(\gamma\mright)\right)\]
  holds.

  Fix $\gamma, \beta \in \alpha$ such that $\gamma \subseteq \beta$, and assume that $\delta \in \gamma$, $\varepsilon \in \alpha$ and $\varepsilon \subseteq \delta$. It follows that $\delta \in \beta$ (and $\delta \in \alpha$ by transitivity), thus by inductive hypothesis $\relpl_\alpha\mleft(\varepsilon\mright)$ holds. We know that $\varepsilon \subseteq \delta$, $\gamma \in \alpha$ and $\delta \in \gamma$, thus the assumption $\alpha \in \PlOrd$ must imply that $\varepsilon \in \gamma$, precisely what we need to conclude that $\relpl_\alpha\mleft(\gamma\mright)$ holds.
\end{proof}

\begin{lemma}
  \label{lem:pl-ord-subset-rel-pl-is-pl-ord}
  For any $\alpha \in \PlOrd$, if $\beta \subseteq \alpha$ and $\relpl_\alpha\mleft(\beta\mright)$ holds, then $\beta \in \PlOrd$.
\end{lemma}

\begin{proof}
  First notice that $\beta \subseteq \alpha$ and $\relpl_\alpha\mleft(\beta\mright)$ implies that $\beta$ is transitive. Consider any $\gamma \in \beta$, $\delta \in \gamma$, we know that $\delta \in \alpha$ since $\alpha \in \Ord$ and also $\delta \subseteq \gamma$ since $\gamma \in \Ord$. It follows from $\relpl_\alpha\mleft(\beta\mright)$ that we must have $\delta \in \beta$.

  Now suppose that $\beta' \in \beta$, $\gamma \subseteq \beta'$ and $\relpl_\beta\mleft(\gamma\mright)$ holds. For any $\delta \in \gamma$ and $\varepsilon \in \alpha$ such that $\varepsilon \subseteq \delta$, we know by transitivity that $\delta \in \beta$, so it follows from $\relpl_\alpha\mleft(\beta\mright)$ that $\varepsilon \in \beta$, and it follows again from $\relpl_\beta\mleft(\gamma\mright)$ that $\varepsilon \in \gamma$. In other words, $\relpl_\alpha\mleft(\gamma\mright)$ holds.

  From the assumption that $\alpha \in \PlOrd$, we then know that $\gamma \in \alpha$. Since also $\beta' \in \beta$ and $\gamma \subseteq \beta'$, it follows from $\relpl_\alpha\mleft(\beta\mright)$ that $\gamma \in \beta$. For any $\delta \in \beta$, we now have $\beta' \in \alpha$, $\gamma \subseteq \beta'$ and $\delta \in \alpha$, so it follows from $\relpl_\alpha\mleft(\gamma\mright)$ that $\beta' \in \delta$ implies $\gamma \in \delta$. We have shown that $\beta \in \PlOrd$.
\end{proof}

\begin{proposition}
  \label{prop:pl-ord-class-trans}
  For any $\alpha \in \PlOrd$, if $\beta \in \alpha$, then $\beta \in \PlOrd$.
\end{proposition}

\begin{proof}
  We have $\relpl_\alpha\mleft(\beta\mright)$ by \autoref{lem:pl-ord-elem-rel-pl}, and we also have $\beta \subseteq \alpha$ by transitivity. Thus we have $\beta \in \PlOrd$ by \autoref{lem:pl-ord-subset-rel-pl-is-pl-ord}.
\end{proof}

\begin{lemma}
  \label{lem:subset-pl-ord-sat-relpl}
  For any $\alpha, \beta \in \PlOrd$, if $\beta \subseteq \alpha$, then $\relpl_\alpha\mleft(\beta\mright)$ holds.
\end{lemma}

\begin{proof}
  Suppose that $\delta \in \beta$, $\varepsilon \in \alpha$ and $\varepsilon \subseteq \delta$. By \autoref{lem:pl-ord-elem-rel-pl}, we know that $\relpl_\alpha\mleft(\varepsilon\mright)$ holds. Since $\beta \subseteq \alpha$, it is easy to check that $\relpl_\beta\mleft(\varepsilon\mright)$ also holds. Then the assumption that $\beta \in \PlOrd$ immediately implies $\varepsilon \in \beta$, precisely what we need to show.
\end{proof}

\begin{proposition}
  \label{prop:pl-ord-closed-under-pl-ord-subset}
  For any $\alpha, \gamma \in \PlOrd$, if $\gamma \subseteq \beta$ for some $\beta \in \alpha$, then $\gamma \in \alpha$.
\end{proposition}

\begin{proof}
  Since $\alpha \in \PlOrd$, it suffices to show that $\relpl_\alpha\mleft(\gamma\mright)$ holds. By we know that $\gamma \subseteq \beta \subseteq \alpha$ by transitivity, thus \autoref{lem:subset-pl-ord-sat-relpl} suffices for the proof.
\end{proof}

\autoref{prop:pl-ord-class-trans} and \autoref{prop:pl-ord-closed-under-pl-ord-subset} together form one direction of the claim that $\PlOrd$ is indeed the class of plump ordinals we defined in \autoref{sec:intro}. We now show that the same criteria are also sufficient:

\begin{proposition}
  \label{prop:pl-ord-crit}
  For any set $x$, we have $x \in \PlOrd$ if both of the following holds:
  \begin{itemize}
    \item for any $\beta \in x$, $\beta \in \PlOrd$;
    \item for any $\gamma \in \PlOrd$, if $\gamma \subseteq \beta$ for some $\beta \in x$, then $\gamma \in x$.
  \end{itemize}
\end{proposition}

\begin{proof}
  The first condition already ensures that $x \subseteq \Ord$. For transitivity, note that for any $\beta \in x$, $\gamma \in \beta$, we know that $\gamma \subseteq \beta$ since $\beta \in \Ord$, and $\gamma \in \PlOrd$ by \autoref{prop:pl-ord-class-trans}, thus the second condition implies that $\gamma \in x$ as well.

  Now suppose that $\beta \in x$, $\gamma \subseteq \beta$ and $\relpl_x\mleft(\gamma\mright)$ holds. Since $\beta \subseteq x$ by transitivity, it follows immediately from definition that $\relpl_\beta\mleft(\gamma\mright)$ also holds. Here $\beta \in \PlOrd$, so by \autoref{lem:pl-ord-subset-rel-pl-is-pl-ord}, $\gamma \in \PlOrd$, thus our assumption already ensures that $\gamma \in x$.

  Additionally, for any $\delta \in x$, suppose that $\beta \in \delta$. Observe that we also have $\relpl_\delta\mleft(\gamma\mright)$ and $\delta \in \PlOrd$, so it follows immediately that $\gamma \in \delta$ as well.
\end{proof}

This allows us to finally conclude that $\PlOrd$ is indeed the class of plump ordinals. This verifies that the notion of plump ordinals is $\Pi_1$-definable in $\mathrm{IKP}$, or $\Delta_0\mleft(\mathcal{P}\mright)$-definable in $\mathrm{IKP}\mleft(\mathcal{P}\mright)$ where the powerset operator $\mathcal{P}$ exists and is treated as a primitive symbol.

Additionally observe that one can combine \autoref{lem:pl-ord-subset-rel-pl-is-pl-ord} and \autoref{lem:subset-pl-ord-sat-relpl} to show that for any $\alpha \in \PlOrd$ and any $\beta \subseteq \alpha$,
\[\beta \in \PlOrd \leftrightarrow \relpl_\alpha\mleft(\beta\mright),\]
where the right-hand side of the equivalence is a $\Delta_0$-formula. Therefore, for any $\alpha \in \PlOrd$, we can define its \emph{plump successor} as
\[\alpha^\pls = \left\{\beta \in \mathcal{P}\mleft(\alpha\mright) : \relpl_\alpha\mleft(\beta\mright)\right\} = \mathcal{P}\mleft(\alpha\mright) \cap \PlOrd\]
if this exists as a set; and we know in $\mathrm{IKP}$ that whenever the powerset $\mathcal{P}\mleft(\alpha\mright)$ exists, the plump successor $\alpha^\pls$ also exists by $\Delta_0$-separation. (We use the standard notation $\alpha^+ = \alpha \cup \left\{\alpha\right\}$ to still denote the \emph{thin successor} for any $\alpha \in \Ord$.)

\begin{lemma}
  \label{lem:pl-successor-still-pl}
  For any $\alpha \in \PlOrd$, if $\alpha^\pls$ exists, then $\alpha^\pls \in \PlOrd$.
\end{lemma}

\begin{proof}
  We just need to check the two criteria in \autoref{prop:pl-ord-crit}: for any $\beta \in \alpha^\pls$, we have $\beta \in \PlOrd$ by definition; for any $\gamma \in \PlOrd$, if $\gamma \subseteq \beta$ for some $\beta \in \alpha^\pls$, then $\gamma \subseteq \alpha$ and $\gamma \in \alpha^\pls$ again by definition.
\end{proof}

\begin{lemma}
  \label{lem:pl-union-still-pl}
  For any set $s \subseteq \PlOrd$, $\bigcup s \in \PlOrd$.
\end{lemma}

\begin{proof}
  We check the two criteria in \autoref{prop:pl-ord-crit}: for any $\beta \in \bigcup s$, we have $\beta \in \alpha$ for some $\alpha \in s \subseteq \PlOrd$. Then $\beta \in \PlOrd$ by \autoref{prop:pl-ord-class-trans}. For any $\gamma \in \PlOrd$, if $\gamma \subseteq \beta$ for some $\beta \in \bigcup s$, then again we have $\beta \in \alpha$ for some $\alpha \in s \subseteq \PlOrd$, and it follows that $\gamma \in \alpha \subseteq \bigcup s$ by \autoref{prop:pl-ord-closed-under-pl-ord-subset}.
\end{proof}

\section{Constructible universe and inner model}

The constructible universe does not synergise well with the usual notion of ordinals in intuitionistic contexts, as the old problem of whether one can prove $\Ord \subseteq L$ in intuitionistic set theory still remains open. Thus, in some sense one could say that the intuitionistic $L$ may fail to be an ``inner model'' of the universe $V$. In this section, we shall show instead that the \emph{plump constructible universe} $L_\pl = \bigcup_{\alpha \in \PlOrd} L_\alpha$ becomes a model of $\mathrm{IKP}$ and already contains all plump ordinals.

\subsection{Axiom of unboundedness}

Before we proceed, one must notice that $\mathrm{IKP}$ does not guarantee the existence of powersets, thus not the existence of plump successors $\alpha^\pls = \mathcal{P}\mleft(\alpha\mright) \cap \PlOrd$ either. Therefore one may wish to move to a stronger base theory when working with plump ordinals. We will examine several additional axioms to commit to here and later in \autoref{subsec:bounding-axioms}.

The first natural assumption to make when studying these ordinals is the axiom $\mathrm{PlUb}$ that the class $\PlOrd$ is unbounded, i.e.
\[\forall \alpha \in \PlOrd \ \exists \beta \in \PlOrd \ \alpha \in \beta.\]

\begin{lemma}
  The following are equivalent:
  \begin{itemize}
    \item $\mathrm{PlUb}$;
    \item For any $\alpha \in \PlOrd$, $\alpha^\pls$ exists as a set.
  \end{itemize}
\end{lemma}

\begin{proof}
  Suppose that $\alpha, \beta \in \PlOrd$ and $\alpha \in \beta$, then for any $\gamma \subseteq \alpha$, by \autoref{prop:pl-ord-closed-under-pl-ord-subset} $\gamma \in \PlOrd$ implies $\gamma \in \beta$. Therefore,
  \[\alpha^\pls = \left\{\gamma \in \beta : \gamma \subseteq \alpha \land \relpl_\alpha\mleft(\gamma\mright)\right\}\]
  is a set by $\Delta_0$-separation.

  The backward direction is trivial as $\alpha \in \alpha^\pls$ by definition.
\end{proof}

We immediately see that $\mathrm{PlUb}$ is independent of $\mathrm{IKP}$:
\begin{proposition}
  \label{prop:ikp-maybe-no-1-pls}
  $\mathrm{IKP}$ alone does not prove $\mathrm{PlUb}$. Specifically, we have $1 \in \PlOrd$ yet $\mathrm{IKP} \nvdash \text{$1^\pls$ exists}$.
\end{proposition}

\begin{proof}
  We first observe that any $\alpha \subseteq 1$ is a plump ordinal: it is trivial that $\alpha \in \Ord$; for any $\beta \in \alpha$, we have $\beta = \varnothing$, thus for any $\gamma \subseteq \beta$, clearly $\gamma = \beta$ and thus $\gamma \in \alpha$ as well. Therefore, $\mathcal{P}\mleft(1\mright) \subseteq \PlOrd$. It follows that $1^\pls$, if exists, is precisely $\mathcal{P}\mleft(1\mright) \cap \PlOrd = \mathcal{P}\mleft(1\mright)$.

  It is conventional to verify that $\mathrm{IKP} \nvdash \text{$\mathcal{P}\mleft(1\mright)$ exists (as a set)}$.
\end{proof}

Now assuming $\mathrm{PlUb}$, we can define the \emph{plump operator} $\left(-\right)^\pl : \Ord \rightarrow \PlOrd$ as the following:
\[\alpha^\pl = \bigcup_{\beta \in \alpha} \left(\beta^\pl\right)^\pls.\]
It follows from \autoref{lem:pl-successor-still-pl} and \autoref{lem:pl-union-still-pl} that this is well-defined and $\alpha^\pl \in \PlOrd$ for every $\alpha \in \Ord$. However, we note that this map is very likely not injective.

\begin{proposition}
  The following are equivalent:
  \begin{itemize}
    \item $\mathrm{PlUb} + \text{$\left(-\right)^\pl : \Ord \rightarrow \PlOrd$ is injective}$;
    \item The axiom of restricted excluded middle, i.e.\ $\forall x \ \forall y \left(x \in y \lor \neg x \in y\right)$.
  \end{itemize}
\end{proposition}

\begin{proof}
  We start with the easier, backward direction. Assuming restricted excluded middle, i.e.\ excluded middle for any $\Delta_0$-formulae as well, we know that the classical proof for the trichotomy property of ordinals goes through. It follows that for any ordinals $\alpha, \beta \in \Ord$ such that $\alpha \subseteq \beta$, we must have $\alpha \in \beta \lor \alpha = \beta$. Thus every ordinal is plump, $\alpha^\pls = \alpha^+$ always exists and $\left(-\right)^\pl$ is simply the identity map.

  For the forward direction, as in the proof of \autoref{prop:ikp-maybe-no-1-pls}, we know that $\mathrm{PlUb}$ implies that $1^\pls = \mathcal{P}\mleft(1\mright)$ exists. For any sets $x, y$, by considering the truth value
  \[\sigma_{x, y} = \left\{\varnothing : x \in y\right\} \in \mathcal{P}\mleft(1\mright),\]
  we know that restricted excluded middle holds if and only if $\forall x \ \forall y \ \sigma_{x, y} \in 2$, if and only if\footnote{The backward direction here depends on the observation that for any $\alpha \in \mathcal{P}\mleft(1\mright)$, $\sigma_{\varnothing, \alpha} = \alpha$.} $\mathcal{P}\mleft(1\mright) = 2$. It suffices to show that $\left(-\right)^\pl$ is injective implies $\mathcal{P}\mleft(1\mright) = 2$.

  We look at $2^\pl$ and $\mathcal{P}\mleft(1\mright)^\pl$. Since $2 \subseteq \mathcal{P}\mleft(1\mright)$, by definition we have $2^\pl \subseteq \mathcal{P}\mleft(1\mright)^\pl$. On the other hand, for any $\alpha \in \mathcal{P}\mleft(1\mright)^\pl$, there exists $\beta \subseteq 1$ such that $\alpha \in \left(\beta^\pl\right)^\pls$, i.e.\ $\alpha \subseteq \beta^\pl \subseteq 1^\pl$. It follows that $\alpha \in \left(1^\pl\right)^\pls \subseteq 2^\pl$. Thus, we have $2^\pl = \mathcal{P}\mleft(1\mright)^\pl$, and the injectivity of $\left(-\right)^\pl$ indeed implies $\mathcal{P}\mleft(1\mright) = 2$ as desired.
\end{proof}

\begin{lemma}
  The following are equivalent:
  \begin{itemize}
    \item $\mathrm{PlUb} + \text{Exponentation}$;
    \item Every set has a powerset.
  \end{itemize}
\end{lemma}

\begin{proof}
  For the forward direction, we again only need from $\mathrm{PlUb}$ the impliation that $1^\pls = \mathcal{P}\mleft(1\mright)$ exists. For any set $x$, suppose that the exponentiation $x^{\mathcal{P}\mleft(1\mright)}$ exists, then we have
  \[\mathcal{P}\mleft(x\mright) = \left\{\left\{y \in x : \varnothing \in f\mleft(y\mright)\right\} : f \in x^{\mathcal{P}\mleft(1\mright)}\right\}\]
  by $\Delta_0$-replacement.

  The backward direction is trivial, since for any $\alpha \in \PlOrd$, if $\mathcal{P}\mleft(\alpha\mright)$ is a set, then so does
  \[\alpha^\pls = \left\{\beta \in \mathcal{P}\mleft(\alpha\mright) : \relpl_\alpha\mleft(\beta\mright)\right\}\]
  by $\Delta_0$-separation.
\end{proof}

\subsection{The plump constructible universe}

We follow the treatment of the intuitionistic constructible universe in Lubarsky \cite{lubarsky93-intuitionistic-l}. Let the constructible universe $L = \bigcup_{\alpha \in \Ord} L_\alpha$ be given by a $\Sigma_1$-formula recursively through
\[L_\alpha = \bigcup_{\beta \in \alpha} \dd\mleft(L_\beta\mright),\]
where $\dd\mleft(X\mright)$ denotes the set of first-order definable subsets of the structure $\tuple{X; \in}$. Each $L_\alpha$ will be a transitive set.

We note the proof that for any $\alpha \in \Ord$, $L_\alpha \cap \Ord = \alpha$ requires the trichotomy property of ordinals, thus classical logic. As pointed out in \cite{lubarsky93-intuitionistic-l}, it remains open whether usual intuitionistic set theories can prove $\Ord \subseteq L$. However, we see that the behaviours of plump ordinals are much simpler to characterise:

\begin{lemma}
  \label{lem:uniform-def-pl-ord-in-l}
  For any $\alpha \in \PlOrd$, we have
  \[\alpha = \left\{x \in L_\alpha : L_\alpha \vDash x \in \PlOrd\right\} \in \dd\mleft(L_\alpha\mright).\]
\end{lemma}

\begin{proof}
  We shall prove this by set induction on $\alpha$. Fix some $\alpha \in \PlOrd$, then the inductive hypothesis says that for any $\beta \in \alpha$,
  \[\beta = \left\{x \in L_\beta : L_\beta \vDash x \in \PlOrd\right\} \in \dd\mleft(L_\beta\mright),\]
  so immediately $\alpha \subseteq L_\alpha$. Notice additionally that $\PlOrd$ is a $\Pi_1$-class, which is downward absolute on transitive sets/classes, thus for any $\beta \in \alpha$, since $\beta \in \PlOrd$, we must also have $L_\alpha \vDash \beta \in \PlOrd$.

  It remains to prove the converse: for any $x \in L_\alpha$ such that $L_\alpha \vDash x \in \PlOrd$, we want to show that $x \in \alpha$. There must exist some $\beta \in \alpha$ such that $x \in \dd\mleft(L_\beta\mright)$. Consider any $y \in x$, we observe that \autoref{lem:pl-ord-elem-rel-pl} and \autoref{lem:pl-ord-subset-rel-pl-is-pl-ord} use no other assumptions than the axiom of set induction, thus their relativisations to any transitive set still hold, and we can relativise \autoref{prop:pl-ord-class-trans}:
  \[L_\alpha \vDash \forall x \in \PlOrd \ \forall y \in x \ y \in \PlOrd.\]
  Here, $\alpha$ is by definition transitive, thus $\beta \subseteq \alpha$ and $L_\beta \subseteq L_\alpha$. By downward absoluteness again, we know that $y \in L_\beta$ satisfies $L_\beta \vDash y \in \PlOrd$. By the inductive hypothesis, this means $y \in \beta$, i.e.\ $x \subseteq \beta$.

  Finally, we observe that \autoref{lem:subset-pl-ord-sat-relpl} also relativises to the transitive set $L_\alpha$, thus we have $L_\alpha \vDash \relpl_\beta\mleft(x\mright)$. However, $\relpl_\beta\mleft(x\mright)$ is a $\Delta_0$-formula, so it also holds in the original universe. Thus, $x \in \PlOrd$, and $\alpha \in \PlOrd$ immediately implies that $x \in a$.
\end{proof}

We shall denote the \emph{plump constructible universe} as $L_\pl = \bigcup_{\alpha \in \PlOrd} L_\alpha$.

\begin{corollary}
  \label{cor:l-preserve-pl-ord}
  Assuming $\mathrm{PlUb}$, then $\PlOrd = \PlOrd^{L_\pl}$.
\end{corollary}

\begin{proof}
  For any $\alpha \in \PlOrd$, by \autoref{lem:uniform-def-pl-ord-in-l}, $\alpha \in \dd\mleft(L_\alpha\mright)$, thus $\alpha \in L_\beta$ for any $\beta \in \PlOrd$ such that $\alpha \in \beta$. We also have $L_\pl \vDash \alpha \in \PlOrd$ by downward absoluteness, thus $\PlOrd \subseteq \PlOrd^{L_\pl}$.

  On the other hand, suppose that $L_\pl \vDash x \in \PlOrd$. Then there is some $\alpha \in \PlOrd$ such that $x \in L_\alpha$. By downward absoluteness, $L_\alpha \vDash x \in \PlOrd$ as well. However, by \autoref{lem:uniform-def-pl-ord-in-l} this means $x \in \alpha$, which implies $x \in \PlOrd$ by \autoref{lem:pl-ord-elem-rel-pl}.
\end{proof}

We have shown that $L_\pl$ contains the same plump ordinals as $V$. We want to proceed to verify that it is also a model of $\mathrm{IKP}$. First note that we cannot show that $\omega$ is a plump ordinal. Therefore, to show $\omega \in L_\pl$, we need the following helpful lemma:

\begin{lemma}
  \label{lem:delta-0-def-omega}
  Let $x$ be any transitive set such that $\omega \subseteq x$, then
  \[\omega = \left\{y \in x : y \in \Ord \land \left(y = \varnothing \lor \exists z \in y \ y = z^+\right) \land \forall z \in y \left(z = \varnothing \lor \exists t \in y \ z = t^+\right)\right\}.\]
\end{lemma}

\begin{proof}
  For the forward direction, we can easily show by induction on $\omega$ such that
  \[\forall n \in \omega \left(n = \varnothing \lor \exists m \in n \ n = m^+\right).\]
  That every $n \in \omega$ additionally satisfies $\forall m \in n \left(m = \varnothing \lor \exists k \in n \ m = k^+\right)$ follows immediately from transitivity.

  For the backward direction, we can show by set induction on $y$ that every $y \in x$ lying in the set on the right-hand side is also in $\omega$: we know that either $y = \varnothing$ or there exists $z \in y$ such that $y = z^+$, and the former case is trivial. In the latter case, we additionally know that either $z = \varnothing$ or there exists $t \in y$ such that $z = t^+$. If this $t$ exists, then $t \in z$ by definition of the successor $t^+$, and for any $s \in z$, we know by transitivity that additionally $s \in y$ and hence either $s = \varnothing$ or there exists $r \in y$ such that $s = r^+$, where by transitivity we can easily check that also $r \in z$. In other words, we have
  \[\left(z = \varnothing \lor \exists t \in z \ z = t^+\right) \land \forall s \in z \left(s = \varnothing \lor \exists r \in z \ s = r^+\right).\]
  By the inductive hypothesis, we must have $z \in \omega$ and hence $y = z^+ \in \omega$ as well.
\end{proof}

\begin{proposition}
  \label{prop:l-pl-ikp-minus-coll}
  Assuming $\mathrm{PlUb}$, then $L_\pl \vDash \mathrm{IKP} - \text{$\Delta_0$-collection} + \mathrm{PlUb}$.
\end{proposition}

\begin{proof}
  Firstly, the axiom of empty set is trivial; the axioms of extensionality and set induction are also trivial since $L_\pl$ is a transitive class. The axioms of union and $\Delta_0$-separation follows from the fact that any $\Delta_0$-definable subset of $L_\alpha$ lies in $\dd\mleft(L_\alpha\mright)$ and hence $L_{\alpha^\pls}$, which exists by the assumption $\mathrm{PlUb}$. Likewise, for any $a \in L_\alpha$, $b \in L_\beta$ where $\alpha, \beta \in \PlOrd$, we know by \autoref{lem:pl-union-still-pl} that $\alpha \cup \beta \in \PlOrd$ with $\left\{a, b\right\} \subseteq L_{\alpha \cup \beta}$. Thus
  \[\left\{a, b\right\} \in \dd\mleft(L_{\alpha \cup \beta}\mright) \subseteq L_{\left(\alpha \cup \beta\right)^\pls}.\]

  For the axiom of strong infinity, we will consider a specific $\Sigma_1$-formulation
  \[\exists \omega \left(\varnothing \in \omega \land \forall x \in \omega \ x^+ \in \omega \land \forall x \in \omega \left(x = \varnothing \lor \exists y \in \omega \ x = y^+\right)\right).\]
  This is simply asserting that a set $\omega$ satisfying some $\Delta_0$-formula exists and the axiom of set induction suffices to show that $\omega$ defined in this way is the unique intersection of all inductive sets. Therefore, by absoluteness of $\Delta_0$-formulae, it suffices to show that $\omega \in L_\pl$.

  Now, each $n \in \omega$ is hereditarily finite, thus uniquely definable by a $\Delta_0$-formula. Therefore, a simple induction on $\omega$ shows that
  \[\forall n \in \omega \ n \in \dd\mleft(L_{n^\pl}\mright),\]
  i.e.\ $\omega \subseteq L_{\omega^\pl}$. By \autoref{lem:delta-0-def-omega}, $\omega$ is definable through $\Delta_0$-separation from any transitive superset, thus $\omega \in \dd\mleft(L_{\omega^\pl}\mright) \subseteq L_{\left(\omega^\pl\right)^\pls}$ as we wanted.

  Finally, $L_\pl \vDash \mathrm{PlUb}$ follows trivially from \autoref{cor:l-preserve-pl-ord} that $\PlOrd^{L_\pl} = \PlOrd$.
\end{proof}

\subsection{Bounding axioms}
\label{subsec:bounding-axioms}

The difficulty of relativising $\Delta_0$-collection to $L_\pl$ lies in the fact that $\PlOrd$ is a $\Pi_1$-class, thus the premise $\forall x \in a \ \exists y \in L_\pl \ \varphi\mleft(x, y\mright)$, where $\varphi\mleft(x, y\mright)$ is $\Delta_0$, fails to be a $\Sigma_1$-formula. This prompts us naturally to consider adding the following axiom scheme
\[\forall x \in a \ \exists \alpha \in \PlOrd \ \varphi\mleft(x, \alpha\mright) \rightarrow \exists \alpha \in \PlOrd \ \forall x \in a \ \exists \beta \in \alpha \ \varphi\mleft(x, \beta\mright),\]
which shall be denoted as $\mathrm{PlB}\Gamma$, read as \emph{$\Gamma$-bounding for plump ordinals}, when $\varphi$ ranges over all $\Gamma$-formulae.

\begin{lemma}
  \label{lem:delta-0-bounding-implies-plub}
  $\mathrm{PlB}\Delta_0$ implies $\mathrm{PlUb}$.
\end{lemma}

\begin{proof}
  For any $\alpha \in \PlOrd$, we trivially have $\forall x \in \left\{\alpha\right\} \ \exists \beta \in \PlOrd \ x = \beta$. Therefore, by $\Delta_0$-bounding, there must exist $\gamma \in \PlOrd$ such that
  \[\forall x \in \left\{\alpha\right\} \ \exists \beta \in \gamma \ x = \beta,\]
  i.e.\ $\alpha \in \gamma$ as needed.
\end{proof}

We now have the following:

\begin{proposition}
  \label{prop:l-preserve-sigma-1-bounding}
  Assuming $\mathrm{PlB}\Sigma_1$, then $L_\pl \vDash \text{$\Sigma_1$-collection} + \mathrm{PlB}\Sigma_1$.
\end{proposition}

\begin{proof}
  Assume $a \in L_\pl$ and we have
  \[\forall x \in a \ \exists y, z \in L_\pl \ \varphi\mleft(x, y, z\mright),\]
  where $\varphi\mleft(x, y, z\mright)$ is $\Delta_0$ and hence absolute. This can be equivalently written as
  \[\forall x \in a \ \exists \alpha \in \PlOrd \ \exists y, z \in L_\alpha \ \varphi\mleft(x, y, z\mright),\]
  where $L_\alpha$ is $\Sigma_1$-definable, so by $\mathrm{PlB}\Sigma_1$, there must exist $\alpha \in \PlOrd$ such that $\forall x \in a \ \exists \beta \in \alpha \ \exists y, z \in L_\beta \ \varphi\mleft(x, y, z\mright)$. It follows that
  \[\forall x \in a \ \exists y \in L_\alpha \ \exists z \in L_\pl \ \varphi\mleft(x, y, z\mright)\]
  where $L_\alpha \in L_\pl$. Thus $L_\pl \vDash \text{$\Sigma_1$-collection}$.

  $L_\pl \vDash \mathrm{PlB}\Sigma_1$ follows from an entirely analogous argument by letting $y$ above range over $\PlOrd$ instead, using the fact that $\mathrm{PlUb}$ holds by \autoref{lem:delta-0-bounding-implies-plub} and thus $\PlOrd^{L_\pl} = \PlOrd$ by \autoref{cor:l-preserve-pl-ord}.
\end{proof}

\begin{corollary}
  \label{cor:l-pl-ikp-inner-model}
  Assuming $\mathrm{PlB}\Sigma_1$, then $L_\pl$ is an inner model (in the sense that it contains precisely the same plump ordinals as $V$) satisfying
  \[\mathrm{IKP} + \mathrm{PlB}\Sigma_1 + V = L_\pl.\]
\end{corollary}

\begin{proof}
  This follows immediately from \autoref{prop:l-pl-ikp-minus-coll}, \autoref{prop:l-preserve-sigma-1-bounding} and \autoref{cor:l-preserve-pl-ord}.
\end{proof}

Observe also that $\mathrm{PlB}\Sigma_1$ is an consequence of $\mathrm{IKP}\mleft(\mathcal{P}\mright)$ (through \autoref{lem:pl-union-still-pl} above), since $\PlOrd$ is a $\Delta_0\mleft(\mathcal{P}\mright)$-class, where the powerset operator is treated as a primitive symbol. However, this does not mean that $\mathrm{IKP}\mleft(\mathcal{P}\mright)$ can also verify that $L_\pl$ is an inner model, since powerset is a more serious obstacle here: we know from \cite{matthews-rathjen24-constructible-universe}*{Theorem 7.12} that even $\mathrm{CZF}\mleft(\mathcal{P}\mright) \nvdash L \vDash \text{Exponentation}$ and I suspect that the same applies to $L_\pl$ easily as well.

\section{Plump ordinal arithmetic}
\label{sec:plump-ord-arith}

\subsection{Basic properties of ordinal arithmetic}

We now work in $\mathrm{IKP} + \mathrm{PlB}\Sigma_1$. Let $\Sigma^\pl$ be the class of formulae with strictly positive occurrences of unbounded $\exists x$ and $\exists x \in \PlOrd$ and no occurrences of $\forall x$ and $\forall x \in \PlOrd$, then $\mathrm{PlB}\Sigma_1$ implies that every $\Sigma^\pl$-formula is equivalent to $\exists \alpha \in \PlOrd \ \varphi\mleft(\alpha\mright)$ where $\varphi$ is $\Sigma_1$, by replacing every inner occurrence of $x \in \PlOrd$ with $\relpl_\alpha\mleft(x\mright)$. This means the usual proof for $\Sigma$-recursion in $\mathrm{IKP}$ carries over for $\Sigma^\pl$-formulae in this system.

Thus following section 7 of Taylor \cite{taylor96-intuitionistic-ordinals}, we can define the usual arithmetic operations on plump ordinals:
\begin{align*}
  \alpha + \beta     & = \alpha \cup \bigcup_{\gamma \in \beta} \left(\alpha + \gamma\right)^{\pl+}, \\
  \alpha \cdot \beta & = \bigcup_{\gamma \in \beta} \left(\alpha \cdot \gamma + \alpha\right),       \\
  \alpha^\beta       & = 1 \cup \bigcup_{\gamma \in \beta} \left(\alpha^\gamma \cdot \alpha\right).
\end{align*}
This will yield different results from the usual ordinal arithmetic, due to the initial step being defined as $\alpha + 1 = \alpha^{\pl+}$. To be clear of ambiguity, in this exposition we will only use plump ordinal arithmetic, and thus will not bother with adding decorators everywhere to distinguish them from usual ordinal arithmetic operations. Also notice that $\beta \in \Ord$ suffices for the following:

\begin{lemma}
  For any $\alpha \in \PlOrd$, $\beta \in \Ord$, we have $\alpha + \beta, \alpha \cdot \beta, \alpha^\beta \in \PlOrd$.
\end{lemma}

\begin{proof}
  This follows from induction on $\beta$, using \autoref{lem:pl-successor-still-pl} and \autoref{lem:pl-union-still-pl}.
\end{proof}

When additionally $\beta \in \PlOrd$, we check that the following basic properties of arithmetic hold. Note that many short proofs here are specifically for plump ordinals and do not work for the larger class $\Ord$ in an intuitionistic context:

\begin{lemma}
  \label{lem:pl-add-inj}
  For any $\alpha, \beta, \gamma \in \PlOrd$, $\alpha + \beta \subseteq \alpha + \gamma$ implies $\beta \subseteq \gamma$. In other words, the class function $F_\alpha : \PlOrd \rightarrow \PlOrd$ given by $\beta \mapsto \alpha + \beta$ is injective.
\end{lemma}

\begin{proof}
  We show this by induction on $\gamma$. Consider any $\delta \in \beta$, then $\alpha + \delta \in \left(\alpha + \delta\right)^{\pl+} \subseteq \alpha + \beta$. By definition, $\alpha \subseteq \alpha + \delta$, thus $\alpha + \delta \in \alpha$ would imply $\alpha \in \alpha$ by plumpness, thus a contradiction. Therefore, we must have $\alpha + \delta \in \left(\alpha + \varepsilon\right)^{\pl+}$ for some $\varepsilon \in \gamma$. This means $\alpha + \delta \subseteq \alpha + \varepsilon$, and thus $\delta \subseteq \varepsilon$ by the inductive hypothesis. Since $\gamma \in \PlOrd$, we must have $\delta \in \gamma$.
\end{proof}

\begin{lemma}
  \label{lem:pl-mul-inj}
  For any $\alpha, \beta, \beta', \gamma \in \PlOrd$, if $\gamma \in \alpha$, then $\alpha \cdot \beta \subseteq \alpha \cdot \beta' + \gamma$ implies $\beta \subseteq \beta'$.
\end{lemma}

\begin{proof}
  We show this by induction on $\beta'$. Consider any $\delta \in \beta$, we know that $\alpha \cdot \delta + \gamma \in \left(\alpha \cdot \delta + \gamma\right)^{\pl+} \subseteq \alpha \cdot \delta + \alpha$. Thus $\alpha \cdot \delta + \gamma \in \alpha \cdot \beta \subseteq \alpha \cdot \beta' + \gamma$.

  Now, it suffices to show, by another induction on $\gamma$ here, that $\alpha \cdot \delta + \gamma \in \alpha \cdot \beta' + \gamma$ implies $\delta \in \beta'$ (for the fixed $\beta'$ above). By the definition of plump ordinal addition, we must have either $\alpha \cdot \delta + \gamma \in \alpha \cdot \beta'$, or $\alpha \cdot \delta + \gamma \in \left(\alpha \cdot \beta' + \gamma'\right)^{\pl+}$ for some $\gamma' \in \gamma$.

  In the first case, there must exist $\varepsilon \in \beta'$ such that $\alpha \cdot \delta + \gamma \in \alpha \cdot \varepsilon + \alpha$, i.e.\ there exists $\gamma' \in \alpha$ such that $\alpha \cdot \delta + \gamma \in \left(\alpha \cdot \varepsilon + \gamma'\right)^{\pl+}$ and hence $\alpha \cdot \delta \subseteq \alpha \cdot \delta + \gamma \subseteq \alpha \cdot \varepsilon + \gamma'$. By the inductive hypothesis of the outer induction, $\delta \subseteq \varepsilon$, thus $\delta \in \beta'$ by plumpness.

  In the second case, we must have $\alpha \cdot \delta + \gamma' \in \left(\alpha \cdot \delta + \gamma'\right)^{\pl+} \subseteq \alpha \cdot \delta + \gamma \subseteq \alpha \cdot \beta' + \gamma'$, and thus $\delta \in \beta'$ by the inductive hypothesis of the inner induction.
\end{proof}

\begin{corollary}
  \label{cor:pl-ord-code-pairs}
  For any $\alpha, \beta \in \PlOrd$, the function $f : \alpha \times \beta \rightarrow \alpha \cdot \beta$ given by $\tuple{\gamma, \delta} \mapsto \alpha \cdot \delta + \gamma$ is injective.
\end{corollary}

\begin{proof}
  We first check that $f$ is indeed well-defined on codomain $\alpha \cdot \beta$: for any $\gamma \in \alpha, \delta \in \beta$, we see that $\alpha \cdot \delta + \gamma \in \left(\alpha \cdot \delta + \gamma\right)^{\pl+} \subseteq \alpha \cdot \delta + \alpha \subseteq \alpha \cdot \beta$.

  Now, suppose that $f\mleft(\gamma, \delta\mright) = f\mleft(\gamma', \delta'\mright)$. Observe that we have $\alpha \cdot \delta \subseteq f\mleft(\gamma, \delta\mright) = \alpha \cdot \delta' + \gamma'$ and thus $\delta \subseteq \delta'$ by \autoref{lem:pl-mul-inj}. Likewise, we have $\delta' \subseteq \delta$ as well, i.e.\ $\delta = \delta'$. This implies $\alpha \cdot \delta = \alpha \cdot \delta'$ and hence furthermore $\gamma = \gamma'$ by \autoref{lem:pl-add-inj}.
\end{proof}

We now have the infrastructure to prove the following claim intuitionistically:

\begin{theorem}
  \label{thm:l-pl-ord-surj}
  There exists class functions $\Theta, I$ on domain $\PlOrd$, such that for any $\alpha \in \PlOrd$, we have $\Theta\mleft(\alpha\mright) \in \PlOrd$ and $I\mleft(\alpha\mright) \subseteq \Theta\mleft(\alpha\mright) \times L_\alpha$ is a surjection from a subset of $\Theta\mleft(\alpha\mright)$ onto $L_\alpha$.
\end{theorem}

\begin{proof}
  We define $\Theta$ and $I$ simultaneously by $\Sigma^\pl$-recursion. Fix some $\alpha \in \PlOrd$, denote $\gamma = \bigcup_{\beta \in \alpha} \Theta\mleft(\beta\mright)$ and let
  \[\Theta\mleft(\alpha\mright) = \alpha \cdot \omega^\pl \cdot \gamma^\omega.\]
  Then $I\mleft(\alpha\mright)$ will precisely contain those $\tuple{\delta, x} \in \Theta\mleft(\alpha\mright) \times L_\alpha$ such that there exists $\beta \in \alpha$, formula $\varphi$ and $\varepsilon_1, \ldots, \varepsilon_n \in \mathrm{dom}\mleft(I\mleft(\beta\mright)\mright) \subseteq \Theta\mleft(\beta\mright)$ such that
  \begin{itemize}
    \item $\delta = \sum_{i = 1}^n \alpha \cdot \omega^\pl \cdot \gamma^{n - i} \cdot \varepsilon_i + \alpha \cdot q^\pl + \beta$, where $q \in \omega$ is the G\"odel numbering of $\varphi$;
    \item $x = \left\{z \in L_\beta : \varphi\mleft(z, I\mleft(\beta\mright)\mleft(\varepsilon_1\mright), \ldots, I\mleft(\beta\mright)\mleft(\varepsilon_n\mright)\mright)\right\} \in \dd\mleft(L_\beta\mright)$.
  \end{itemize}
  Repeated applications of \autoref{cor:pl-ord-code-pairs} shows immediately that if $\tuple{\delta, x}, \tuple{\delta', x'} \in I\mleft(\alpha\mright)$ with $\delta = \delta'$, then $x$ and $x'$ are defined through exactly the same parameters, thus $I\mleft(\alpha\mright)$ is indeed a well-defined function. The other desired properties are all straightforward to verify.
\end{proof}

\subsection{Strong incomparables and coding}
\label{subsec:incom-code}

We say that two plump ordinals $\alpha, \beta \in \PlOrd$ are \emph{strongly incomparable}
\footnote{The adjective ``strongly'' is to distinguish the notion from that of weak incomparability, i.e.\ $\neg \alpha \in \beta^+ \land \neg \beta \in \alpha^+$, which is useful in other contexts (c.f.\ \cite{wang-sw25-antichain-pre}). Only strong incomparability will be relevant to the discussion of this exposition.} if neither is a subset of the other, i.e.
\[\neg \alpha \subseteq \beta \land \neg \beta \subseteq \alpha,\]
which is equivalent to saying $\alpha \not\in \beta^\pls \land \beta \not\in \alpha^\pls$. Moreover, let $f : s \rightarrow \PlOrd$ be an arbitrary function on some domain $s$, we say that $f$ is \emph{pairwise (strongly) incomparable} if the following holds:
\[\forall \gamma, \gamma' \in s \left(f\mleft(\gamma'\mright) \subseteq f\mleft(\gamma\mright) \rightarrow \gamma = \gamma'\right).\]

Now, arithmetic properties in the previous section provide us a convenient way to encode a pairwise incomparable function on plump ordinals. More specifically, we show that:

\begin{proposition}
  \label{prop:incom-func-as-pl-ord}
  Any pairwise incomparable function on plump ordinals can be encoded by a single plump ordinal. In other words, for any $\alpha, \beta \in \PlOrd$, if $f : s \rightarrow \beta$ is pairwise incomparable on domain $s \subseteq \alpha$, then there exists some $\sigma \in \PlOrd$ with a formula $\varphi\mleft(x, y; \alpha, \beta, \sigma\mright)$ such that
  \[\left(\gamma \in s \land f\mleft(\gamma\mright) = \delta\right) \leftrightarrow \varphi\mleft(\gamma, \delta; \alpha, \beta, \sigma\mright).\]
\end{proposition}

\begin{proof}
  We take
  \[\sigma = \bigcup_{\gamma \in s} \left(\alpha \cdot f\mleft(\gamma\mright) + \gamma\right)^\pls,\]
  and define $\varphi\mleft(x, y; \alpha, \beta, \sigma\mright)$ as the formula
  \[x \in \alpha \land y \in \beta \land \exists \tau \in \sigma \left(\tau = \alpha \cdot y + x \land \forall \tau' \in \sigma \left(\tau \subseteq \tau' \rightarrow \tau = \tau'\right)\right).\]

  For the forward direction, we check for any $\gamma \in s$ that indeed $\alpha \cdot f\mleft(\gamma\mright) + \gamma \in \sigma$. Additionally, for any $\tau \in \sigma$, there must exist $\gamma' \in s$ such that $\tau \subseteq \alpha \cdot f\mleft(\gamma'\mright) + \gamma'$. Now, if $\alpha \cdot f\mleft(\gamma\mright) + \gamma \subseteq \tau$, i.e.
  \[\alpha \cdot f\mleft(\gamma\mright) \subseteq \alpha \cdot f\mleft(\gamma\mright) + \gamma \subseteq \tau \subseteq \alpha \cdot f\mleft(\gamma'\mright) + \gamma',\]
  then by \autoref{lem:pl-mul-inj}, we must also have $f\mleft(\gamma\mright) \subseteq f\mleft(\gamma'\mright)$, which by pairwise incomparability implies $\gamma = \gamma'$. It then follows trivially that $\tau = \alpha \cdot f\mleft(\gamma\mright) + \gamma$.

  For the backward direction, assume that $\varphi\mleft(\gamma, \delta; \alpha, \beta, \sigma\mright)$ holds. Then we must have $\alpha \cdot \delta + \gamma \in \sigma$, i.e.\ $\alpha \cdot \delta + \gamma \subseteq \alpha \cdot f\mleft(\gamma'\mright) + \gamma'$ for some $\gamma' \in s$, which also implies $\alpha \cdot \delta + \gamma = \alpha \cdot f\mleft(\gamma'\mright) + \gamma'$ by stipulation in $\varphi$. Since also $\gamma, \delta \in \PlOrd$ and $\gamma \in \alpha$, it follows from \autoref{cor:pl-ord-code-pairs} that $\delta = f\mleft(\gamma'\mright)$ and $\gamma = \gamma'$. Thus we have $\gamma \in s \land \delta = f\mleft(\gamma\mright)$ as desired.
\end{proof}

Since plump ordinals are always in $L_\pl$ by \autoref{lem:uniform-def-pl-ord-in-l}, we can show the following result:

\begin{theorem}
  \label{thm:l-pl-code-power}
  Suppose that $s \subseteq \PlOrd$ and $f : s \rightarrow \PlOrd$ is pairwise incomparable, then $\mathcal{P}\mleft(s\mright) \subseteq L_\pl$.
\end{theorem}

\begin{proof}
  Let $\alpha = \bigcup_{\beta \in s} \beta^\pls$ so that $s \subseteq \alpha$ and let $x \subseteq s$ be any subset. We consider
  \[\sigma_x = \bigcup_{\beta \in x} f\mleft(\beta\mright)^{\pl+},\]
  then $\sigma_x$ is a plump ordinal, thus $\sigma_x \in L_\pl$ by \autoref{lem:uniform-def-pl-ord-in-l}. We can then define the set
  \[s_x = \left\{\beta \in \alpha : \beta \in s \land f\mleft(\beta\mright) \in \sigma_x\right\},\]
  which is a set in $L_\pl$ because the condition $\beta \in s \land f\mleft(\beta\mright) \in \sigma_x$ can be written as $\exists \gamma \in \sigma_x \ \varphi\mleft(\beta, \gamma\mright)$ through the $\Delta_0$-formula $\varphi\mleft(x, y\mright)$ (except containing functions defined through $\Sigma^\pl$-recursion) with plump parameters given in \autoref{prop:incom-func-as-pl-ord}. Obviously, for any $\beta \in x$, we have $f\mleft(\beta\mright) \in \sigma_x$, thus $\beta \in s_x$. On the other hand, for any $\beta \in s_x$ we have $\beta \in s$ and some $\gamma \in x$ such that $f\mleft(\beta\mright) \in f\mleft(\gamma\mright)^{\pl+}$, i.e.\ $f\mleft(\beta\mright) \subseteq f\mleft(\gamma\mright)$, thus $\beta = \gamma \in x$ by incomparability. Therefore, we indeed have $x = s_x \in L_\pl$.
\end{proof}

\begin{corollary}
  \label{cor:many-pl-incom-v-eq-l-pl}
  Suppose that for any $\alpha \in \PlOrd$, there exists a pairwise incomparable function $f : \alpha \rightarrow \PlOrd$, then we have $V = L_\pl$.
\end{corollary}

\begin{proof}
  By set induction, it suffices to show that for any set $x \subseteq L_\pl$, we also have $x \in L_\pl$. Now, by $\mathrm{PlB}\Sigma_1$, we can find some $\alpha \in \PlOrd$ such that $x \subseteq L_\alpha$. Consider the plump ordinal $\Theta\mleft(\alpha\mright)$ given in \autoref{thm:l-pl-ord-surj}, then \autoref{thm:l-pl-code-power} implies that $\mathcal{P}\mleft(\Theta\mleft(\alpha\mright)\mright) \subseteq L_\pl$. Specifically, the following set
  \[\sigma = \left\{\beta \in \Theta\mleft(\alpha\mright) : \exists y \in x \tuple{\beta, y} \in I\mleft(\alpha\mright)\right\} \in L_\pl.\]

  Now, we work inside $L_\pl$, which is a model of $\mathrm{IKP} + \mathrm{PlB}\Sigma_1$ by \autoref{cor:l-pl-ikp-inner-model}, and can construct the same $I\mleft(\alpha\mright)$ through $\Sigma^\pl$-recursion. It follows that
  \[s = \left\{I\mleft(\alpha\mright)\mleft(\beta\mright) : \beta \in \sigma\right\} \in L_\pl.\]
  For any $y \in x$, there exists $\beta \in \Theta\mleft(\alpha\mright)$ such that $I\mleft(\alpha\mright)\mleft(\beta\mright) = y$ by stipulations in \autoref{thm:l-pl-ord-surj} and $\beta \in \sigma$ by its definition, thus $x \subseteq s$. On the other hand, since $I\mleft(\alpha\mright)$ is a function, for any $\beta \in \sigma$, $I\mleft(\alpha\mright)\mleft(\beta\mright) \in x$ by definition of $\sigma$. Therefore, indeed $x = s \in L_\pl$.
\end{proof}

\section{An application in forcing models}

\autoref{thm:powerset-plump-in-l} and \autoref{cor:many-pl-incom-v-eq-l-pl} are one example how in various occasions related to $L$, plump ordinals and their arithmetic can be more useful than their counterparts on the full class $\Ord$ of ordinals. For comparison, the incomparable phenomenon in \autoref{subsec:incom-code} can manifest among arbitrary (non-plump) ordinals with much less effort (c.f.\ \cite{wang-sw25-antichain-pre}), but since whether $\Ord \subseteq L$ remains an open problem in intuitionistic set theories, it can be difficult to relate definability with ordinal parameters to definability in $L$.

In this final section, we will see a more concrete example where we utilise \autoref{lem:uniform-def-pl-ord-in-l} and an explicit coding function as considered in \autoref{subsec:incom-code} to automatically encode large swathes of ``non-constructive sets'' into $L$ in an intuitionistic context. We shall work in the stronger, classical theory $\mathrm{ZFC}$ and put an arbitrary classical set $x$ (and all its subsets) inside $L$ through Heyting forcing.

\subsection{Heyting forcing on generalised Cantor space}

Working in $\mathrm{ZFC}$, let $\kappa$ be any infinite cardinal and consider the generalised Cantor topology on $2^\kappa$ generated by basic open sets
\[U_f = \left\{g \in 2^\kappa : f \subseteq g\right\}\]
for any $f : \alpha \rightarrow 2$ with $\alpha \in \kappa$. Following Bell \cite{bell05-boolean-valued-models}*{Chapter 0}, the open sets of this topology form a complete Heyting algebra, which we shall denote $\mathbb{H}_\kappa$.

For any $\alpha, \beta \in \kappa$ and $f : \alpha \rightarrow 2$, we pick an element in the space $2^\kappa$ to encode the ordered pair $\tuple{\beta, f}$; we denote it as $f \oplus \beta : \kappa \rightarrow 2$, given by
\[\mleft(f \oplus \beta\mright)\mleft(\gamma\mright) = \left\{\begin{aligned}
     & f\mleft(\gamma\mright) &  & \text{if $\gamma \in \alpha$},                                \\
     & 1                      &  & \text{if $\gamma = \alpha$ or $\gamma = \alpha + 1 + \beta$}, \\
     & 0                      &  & \text{otherwise}.
  \end{aligned}\right.\]
For any $\alpha \in \kappa$, we denote
\[D_\alpha = \left\{f \oplus \alpha : \beta \in \kappa, f : \beta \rightarrow 2\right\} \subseteq 2^\kappa,\]
which is clearly dense in $2^\kappa$. Observe that:
\begin{lemma}
  \label{lem:d-alpha-disjoint}
  For any $\alpha_1, \alpha_2, \beta_1, \beta_2 \in \kappa$ and $f_1 : \alpha_1 \rightarrow 2$, $f_2 : \alpha_2 \rightarrow 2$, we have
  \[f_1 \oplus \beta_1 = f_2 \oplus \beta_2 \rightarrow \left(\beta_1 = \beta_2 \land f_1 = f_2\right).\]
  In other words, the dense sets $D_\alpha$ are disjoint for different $\alpha \in \kappa$.
\end{lemma}

\begin{proof}
  By definition of $\oplus$, the last $\gamma \in \kappa$ such that $\mleft(f_1 \oplus \beta_1\mright)\mleft(\gamma\mright) = 1$ is $\gamma = \alpha_1 + 1 + \beta_1$ and the second last is $\gamma = \alpha_1$. The same holds for $f_2 \oplus \beta_2$, thus we know that
  \[\alpha_1 = \alpha_2 \quad \text{and} \quad \alpha_1 + 1 + \beta_1 = \alpha_2 + 1 + \beta_2.\]
  It follows that $\beta_1 = \beta_2$ and also $f_1 = f_2$ immediately.
\end{proof}

Using choice, we can fix an ordinal $\lambda$ and a bijection $\rho : \lambda \rightarrow 2^{<\kappa}$ into the set of all functions $f : \alpha \rightarrow 2$, $\alpha \in \kappa$ that respects the well-ordering by inclusion on the domains of these functions, i.e.\ for any $\alpha, \beta \in \lambda$,
\[\alpha \in \beta \rightarrow \mathrm{dom}\mleft(\rho\mleft(\alpha\mright)\mright) \subseteq \mathrm{dom}\mleft(\rho\mleft(\beta\mright)\mright).\]
For any $\alpha \in \kappa$, let $\rho \oplus \alpha$ denote the map $\beta \mapsto \rho\mleft(\beta\mright) \oplus \alpha$, which is a bijection $\lambda \rightarrow D_\alpha$. We observe that:
\begin{lemma}
  \label{lem:d-alpha-initial-segments-closed}
  For any $\alpha \in \kappa$, $\beta \in \lambda$, the image
  \[\mleft(\rho \oplus \alpha\mright)``\beta = \left\{\rho\mleft(\gamma\mright) \oplus \alpha : \gamma \in \beta\right\}\]
  is closed.
\end{lemma}

\begin{proof}
  Let $\sigma = \mathrm{dom}\mleft(\rho\mleft(\beta\mright)\mright) + 1 + \alpha$. By our stipulation of $\rho$, for any $\gamma \in \beta$ we have $\mathrm{dom}\mleft(\rho\mleft(\gamma\mright)\mright) \subseteq \mathrm{dom}\mleft(\rho\mleft(\beta\mright)\mright)$, thus for any $f \in \mleft(\rho \oplus \alpha\mright)``\beta$ it follows that $f\mleft(\gamma\mright) = 0$ for any $\gamma > \sigma$. It suffices to show that any $S \subseteq 2^\kappa$ with this property is closed.

  We consider some arbitrary $g \in 2^\kappa \setminus S$; there are two cases: either there exists some $\gamma > \sigma$ such that $g\mleft(\gamma\mright) = 1$, or there is no such $\gamma$. In the first case, clearly $g \in U_{g \upharpoonright \left(\gamma + 1\right)}$ and $U_{g \upharpoonright \left(\gamma + 1\right)} \cap S = \varnothing$. In the second case, since $g \not\in S$, for every $f \in S$ there must be some $\gamma \leq \sigma$ such that $f\mleft(\gamma\mright) \neq g\mleft(\gamma\mright)$. Thus $g \in U_{g \upharpoonright \left(\sigma + 1\right)}$ and $U_{g \upharpoonright \left(\sigma + 1\right)} \cap S = \varnothing$. Therefore $S$ is closed.
\end{proof}

Now, as in Bell \cite{bell05-boolean-valued-models}*{Chapter 8}, we consider Heyting forcing over the class of names
\[V^{\mathbb{H}_\kappa}_\alpha = \left\{f : \text{$f$ is a partial function} \bigcup_{\beta \in \alpha} V^{\mathbb{H}_\kappa}_\beta \rightharpoonup \mathbb{H}_\kappa\right\},\]
and finally $V^{\mathbb{H}_\kappa} = \bigcup_{\alpha \in \Ord} V^{\mathbb{H}_\kappa}_\alpha$. The valuation of a formula is denoted as $\left\llbracket\varphi\right\rrbracket \in \mathbb{H}_\kappa$ (whose definition is standard and we will conveniently omit here) and we write $V^{\mathbb{H}_\kappa} \vDash \varphi$ when
\[\left\llbracket\varphi\right\rrbracket = \mathbf{1}_{\mathbb{H}_\kappa} = 2^\kappa.\]
Naturally, $V^{\mathbb{H}_\kappa} \vDash \mathrm{IZF}$.

By \autoref{lem:d-alpha-initial-segments-closed}, we can consider the names
\[\dot{\lambda}_\alpha = \left\{\tuple{\check{\beta}, 2^\kappa \setminus \mleft(\rho \oplus \alpha\mright)``\beta} : \beta \in \lambda\right\}\]
for any $\alpha \in \kappa$. We verify that
\begin{lemma}
  For any $\alpha \in \kappa$, $V^{\mathbb{H}_\kappa} \vDash \dot{\lambda}_\alpha \in \Ord \land \dot{\lambda}_\alpha \subseteq \check{\lambda}$. Also for any $\alpha \neq \beta \in \kappa$,
  \[V^{\mathbb{H}_\kappa} \vDash \neg \dot{\lambda}_\alpha \subseteq \dot{\lambda}_\beta.\]
\end{lemma}

\begin{proof}
  It is trivial to check that $\left\llbracket \dot{\lambda}_\alpha \subseteq \check{\lambda} \right\rrbracket = \mathbf{1}_{\mathbb{H}_\kappa}$ and it follows immediate that $V^{\mathbb{H}_\kappa} \vDash \forall \beta \in \dot{\lambda}_\alpha \ \text{$\beta$ is transitive}$. It remains to evaluate that
  \begin{align*}
    \left\llbracket \forall \beta \in \dot{\lambda}_\alpha \ \forall \gamma \in \beta \ \gamma \in \dot{\lambda}_\alpha \right\rrbracket & = \bigwedge_{\beta \in \lambda} \left(2^\kappa \setminus \mleft(\rho \oplus \alpha\mright)``\beta \Rightarrow \bigwedge_{\gamma \in \beta} \left\llbracket \check{\gamma} \in \dot{\lambda}_\alpha \right\rrbracket\right)                                    \\
                                                                                                                                         & = \bigwedge_{\beta \in \lambda} \left(2^\kappa \setminus \mleft(\rho \oplus \alpha\mright)``\beta \Rightarrow \bigwedge_{\gamma \in \beta} \left(2^\kappa \setminus \mleft(\rho \oplus \alpha\mright)``\gamma\right)\right) = \mathbf{1}_{\mathbb{H}_\kappa},
  \end{align*}
  since $\mleft(\rho \oplus \alpha\mright)``\gamma \subseteq \mleft(\rho \oplus \alpha\mright)``\beta$ for $\gamma \in \beta$.

  On the other hand, for any $\alpha \neq \beta \in \kappa$, we have
  \begin{align*}
    \left\llbracket \dot{\lambda}_\alpha \subseteq \dot{\lambda}_\beta \right\rrbracket & = \bigwedge_{\gamma \in \lambda} \left(2^\kappa \setminus \mleft(\rho \oplus \alpha\mright)``\gamma \Rightarrow 2^\kappa \setminus \mleft(\rho \oplus \beta\mright)``\gamma\right)          \\
                                                                                        & = \bigwedge_{\gamma \in \lambda} \left(2^\kappa \setminus \mleft(\rho \oplus \beta\mright)``\gamma\right) = \left(2^\kappa \setminus D_\beta\right)^\circ = \mathbf{0}_{\mathbb{H}_\kappa},
  \end{align*}
  where we use the fact that $\mleft(\rho \oplus \alpha\mright)``\gamma$ and $\mleft(\rho \oplus \beta\mright)``\gamma$ are disjoint by \autoref{lem:d-alpha-disjoint}, thus $2^\kappa \setminus \mleft(\rho \oplus \alpha\mright)``\gamma \Rightarrow 2^\kappa \setminus \mleft(\rho \oplus \beta\mright)``\gamma = 2^\kappa \setminus \mleft(\rho \oplus \beta\mright)``\gamma$.
\end{proof}

Finally, we combine these into the canonical name
\[\dot{\Lambda} = \left\{\tuple{\check{\alpha}, \dot{\lambda}_\alpha}^\bullet : \alpha \in \kappa\right\}^\bullet,\]
such that $V^{\mathbb{H}_\kappa}$ thinks that $\dot{\Lambda}$ is a function $\check{\kappa} \rightarrow \Ord$ that maps each $\check{\alpha}$ to $\dot{\lambda}_\alpha$. It follows easily that
\begin{corollary}
  \label{cor:thin-incom-coding-func}
  $V^{\mathbb{H}_\kappa} \vDash \forall \alpha, \beta \in \check{\kappa} \left(\dot{\Lambda}\mleft(\alpha\mright) \subseteq \dot{\Lambda}\mleft(\beta\mright) \rightarrow \alpha = \beta\right)$.
\end{corollary}

\subsection{Thin ordinals and incomparable coding}

In the process above we built specific forcing names and it's natural there to work with classical ordinals such as $\check{\lambda}$ and their subsets, but these are apparently not plump. However, we may say that an ordinal $\alpha \in \Ord$ is \emph{thin} if for any $\beta, \gamma \in \alpha$,
\[\gamma \subseteq \beta \rightarrow \left(\gamma \in \beta \lor \gamma = \beta\right).\]
It is easy to verify that the forcing model $V^{\mathbb{H}_\kappa}$ recognises any classical $\check{\lambda}$ is thin. In other words, \autoref{cor:thin-incom-coding-func} concerns a function $\dot{\Lambda} : \check{\kappa} \rightarrow \Ord \cap \mathcal{P}\mleft(\check{\lambda}\mright)$ for some thin ordinals $\check{\kappa}, \check{\lambda} \in \Ord$ in $V^{\mathbb{H}_\kappa}$.

From here onwards, we will reason entirely within $V^{\mathbb{H}_\kappa}$ as a model of $\mathrm{IZF}$ to manipulate this distinguished set $\dot{\Lambda}$. We first observe that
\begin{lemma}[$\mathrm{IKP} + \mathrm{PlUb}$]
  \label{lem:pl-subset-retract}
  Let $\alpha \in \Ord$ be thin and $\beta, \gamma \in \mathcal{P}\mleft(\alpha\mright) \cap \Ord$. Then
  \[\gamma \subseteq \beta \leftrightarrow \gamma^\pl \subseteq \beta^\pl.\]
\end{lemma}

\begin{proof}
  The forward direction is trivial from the definition of $\left(-\right)^\pl$; for the backward direction, we prove by simultaneous set induction on the parameters $\beta, \gamma$. Assume that $\gamma^\pl \subseteq \beta^\pl$ and take any $\delta \in \gamma$. We have $\delta^\pl \in \left(\delta^\pl\right)^\pls \subseteq \gamma^\pl$ and hence $\delta^\pl \in \beta^\pl$. In other words, there is some $\varepsilon \in \beta$ such that $\delta^\pl \in \left(\varepsilon^\pl\right)^\pls$, i.e.\ $\delta^\pl \subseteq \varepsilon^\pl$. Now, we still have $\delta, \varepsilon \subseteq \alpha$ by transitivity, so the inductive hypothesis implies $\delta \subseteq \varepsilon$, and by thinness of $\alpha$ we further have $\delta \in \varepsilon \lor \delta = \varepsilon$. Thus $\delta \in \beta$ as desired.
\end{proof}

We define $\dot{\Lambda}^\pl = \left\{\tuple{\alpha^\pl, \beta^\pl} : \tuple{\alpha, \beta} \in \dot{\Lambda}\right\}$. Observe that for any $\tuple{\alpha, \beta} \in \dot{\Lambda}$, we have $\alpha \subseteq \check{\kappa}$ and $\beta \subseteq \check{\lambda}$ for thin ordinals $\check{\kappa}, \check{\lambda} \in \Ord$. Thus by \autoref{lem:pl-subset-retract} we first have
\[\alpha^\pl = \beta^\pl \ \ \text{implies} \ \ \alpha = \beta, \ \ \text{implies} \ \ \gamma = \delta, \ \ \text{implies} \ \ \gamma^\pl = \delta^\pl\]
for any $\tuple{\alpha, \gamma}, \tuple{\beta, \delta} \in \dot{\Lambda}$, i.e.\ $\dot{\Lambda}^\pl$ is a well-defined function on some domain $s = \left\{\alpha^\pl : \alpha \in \check{\kappa}\right\}$. Combining \autoref{lem:pl-subset-retract} with \autoref{cor:thin-incom-coding-func} we also have
\[\dot{\Lambda}^\pl\mleft(\alpha^\pl\mright) \subseteq \dot{\Lambda}^\pl\mleft(\beta^\pl\mright) \ \ \text{implies} \ \ \dot{\Lambda}\mleft(\alpha\mright) \subseteq \dot{\Lambda}\mleft(\beta\mright), \ \ \text{implies} \ \ \alpha = \beta, \ \ \text{implies} \ \ \alpha^\pl = \beta^\pl\]
for any $\alpha, \beta \in \check{\kappa}$, i.e.\ $\dot{\Lambda}^\pl : s \rightarrow \PlOrd$ is a pairwise incomparable function.

\begin{theorem}
  \label{thm:powerset-plump-in-l}
  $V^{\mathbb{H}_\kappa} \vDash \mathcal{P}\mleft(s\mright) \in L_\pl$ where the set $s = \left\{\alpha^\pl : \alpha \in \check{\kappa}\right\}$.
\end{theorem}

\begin{proof}
  This follows immediately from \autoref{thm:l-pl-code-power} on $\dot{\Lambda}^\pl$.
\end{proof}

We can now proceed with the main target of this example, i.e.\ showing that some $V^{\mathbb{H}_\kappa} \vDash \mathcal{P}\mleft(\check{x}\mright) \in L_\pl$ for an arbitrary, fixed $x$:

\begin{lemma}[$\mathrm{ZFC}$]
  \label{lem:arith-pl-commute}
  For any $\alpha, \beta \in \Ord$, let $\gamma = \alpha + \beta$ and $\delta = \alpha \cdot \beta$ as in classical ordinal arithmetic\footnote{Technically speaking, in a classical set theory we have $\PlOrd = \Ord$ and plump ordinal arithmetic coincides with usual ordinal arithmetic. Hence we still abide by the stipulations in \autoref{sec:plump-ord-arith} and are not abusing notations here.}, then
  \[V^{\mathbb{H}_\kappa} \vDash \check{\gamma}^\pl = \check{\alpha}^\pl + \check{\beta}^\pl \quad \text{and} \quad \check{\delta}^\pl = \check{\alpha}^\pl \cdot \check{\beta}^\pl\]
  for any infinite cardinal $\kappa$.
\end{lemma}

\begin{proof}
  This follows from simple induction on $\beta$, using the basic facts in $\mathrm{IKP} + \mathrm{PlUb}$ that
  \begin{enumerate}[label=(\roman*)]
    \item For any $\alpha \in \Ord$, $\left(\alpha^+\right)^\pl = \left(\alpha^\pl\right)^\pls$;
    \item For any $s \subseteq \Ord$, $\left(\bigcup s\right)^\pl = \bigcup_{\alpha \in s} \alpha^\pl$. \qedhere
  \end{enumerate}
\end{proof}

\begin{proposition}[$\mathrm{ZFC}$]
  \label{prop:force-set-in-l}
  Fix any set $x$, then for any large enough cardinal $\kappa$ we have
  \[V^{\mathbb{H}_\kappa} \vDash \check{x} \in L_\pl.\]
\end{proposition}

\begin{proof}
  Using choice, we fix a cardinal $\lambda$ with a surjection $f : \lambda \rightarrow \mathrm{trcl}\mleft(\mleft\{x\mright\}\mright)$ onto the transitive closure of $\left\{x\right\}$, and we can assume without loss of generality that $f\mleft(\varnothing\mright) = x$. Let $\kappa$ be the cardinal successor of $\lambda$, and consider the set
  \[\tau = \left\{\lambda \cdot \alpha + \beta : \alpha, \beta \in \lambda, f\mleft(\alpha\mright) \in f\mleft(\beta\mright)\right\} \subseteq \kappa.\]

  We now work inside $V^{\mathbb{H}_\kappa}$ with the names $\check{f}$ and $\check{\tau}$. We construct $\tilde{\tau} = \left\{\alpha^\pl : \alpha \in \check{\tau}\right\}$ and clearly $\tilde{\tau} \subseteq \left\{\alpha^\pl : \alpha \in \check{\kappa}\right\}$. Thus by \autoref{thm:powerset-plump-in-l}, we have $\tilde{\tau} \in L_\pl$.

  We define
  \[g = \left\{\tuple{\alpha^\pl, y} : \tuple{\alpha, y}\in \check{f}\right\},\]
  then $V^{\mathbb{H}_\kappa}$ recognised that $g$ is a well-defined function on domain $s = \left\{\alpha^\pl : \alpha \in \check{\lambda}\right\}$ by \autoref{lem:pl-subset-retract}. By \autoref{lem:arith-pl-commute}, we also have
  \[V^{\mathbb{H}_\kappa} \vDash \tilde{\tau} = \left\{\check{\lambda}^\pl \cdot \alpha + \beta : \alpha, \beta \in s, g\mleft(\alpha\mright) \in g\mleft(\beta\mright)\right\}.\]

  Now, since also $s \in L_\pl$ trivially by \autoref{thm:powerset-plump-in-l}, we can define inside $L_\pl$ a partial function $h$ on domain $s$ satisfying the recursive condition
  \[h\mleft(\alpha\mright) = \left\{h\mleft(\beta\mright) : \beta \in s, \check{\lambda}^\pl \cdot \beta + \alpha \in \tilde{\tau}\right\}\]
  if $h\mleft(\beta\mright)$ exists for all such $\beta$, through a $\Sigma^\pl$-formula. Using the fact that $L_\pl$ is a model of $\mathrm{IKP} + \mathrm{PlB}\Sigma_1$, we can prove by set induction on $u$ that, for any set $u$, if $g\mleft(\alpha\mright) = u$, then $h\mleft(\alpha\mright)$ exists and equals $u$ as well. Specifically, we know that $g\mleft(\varnothing\mright) = \check{x}$ and $h \in L_\pl$, thus $\check{x} = h\mleft(\varnothing\mright) \in L_\pl$ as desired.
\end{proof}

\begin{corollary}[$\mathrm{ZFC}$]
  Fix any set $x$, then for any large enough cardinal $\kappa$ we have
  \[V^{\mathbb{H}_\kappa} \vDash \mathcal{P}\mleft(\check{x}\mright) \in L_\pl.\]
  Since $L_\pl \subseteq L$, we have $V^{\mathbb{H}_\kappa} \vDash \mathcal{P}\mleft(\check{x}\mright) \in L$ as well.
\end{corollary}

\begin{proof}
  Using choice, we fix a cardinal $\lambda$ with an injection $f : x \rightarrow \lambda$. By \autoref{prop:force-set-in-l}, we can choose $\kappa \geq \lambda$ large enough so that $V^{\mathbb{H}_\kappa} \vDash \check{f} \in L_\pl$. Since \autoref{cor:l-preserve-pl-ord} guarantees that $L_\pl$ is correct about plump ordinals, we can define inside $L_\pl$ the function $\tilde{f} : \check{x} \rightarrow \PlOrd$ as
  \[\tilde{f}\mleft(z\mright) = \check{f}\mleft(z\mright)^\pl,\]
  so that the range of $\tilde{f}$ is a subset of $s = \left\{\alpha^\pl : \alpha \in \check{\kappa}\right\}$.

  Now, observe it suffices to show in $V^{\mathbb{H}_\kappa}$ that $\mathcal{P}\mleft(\check{x}\mright) \subseteq L_\pl$, as $\mathcal{P}\mleft(\check{x}\mright) \in L_\pl$ follows by collection. Consider any subset $y \subseteq \check{x}$, then by \autoref{thm:powerset-plump-in-l} we have
  \[u = \left\{\tilde{f}\mleft(z\mright) : z \in y\right\} \in L_\pl\]
  since $u \subseteq s$. It follows that $v = \left\{z \in \mathrm{dom}\mleft(\tilde{f}\mright) : \tilde{f}\mleft(z\mright) \in u\right\} \in L_\pl$. Notice that $\check{f}$ is trivially injective, and thus so is $\tilde{f}$ by \autoref{lem:pl-subset-retract}. Hence we have $y = v \in L_\pl$ as desired.
\end{proof}

\bibmain

\end{document}